\documentclass[12pt]{amsproc}
\usepackage{amssymb}
\usepackage{amsmath}
\usepackage{amsfonts}
\usepackage{geometry}
\usepackage{graphicx}
\providecommand{\U}[1]{\protect\rule{.1in}{.1in}}
\theoremstyle{plain}

\newtheorem{lemma}{Lemma}

\newtheorem{proposition}{Proposition}
\newtheorem{remark}{Remark}

\newtheorem{theorem}{Theorem}
\numberwithin{equation}{section}

\newcommand{\eps}{\epsilon}
\def\Xint#1{\mathchoice
{\XXint\displaystyle\textstyle{#1}}%
{\XXint\textstyle\scriptstyle{#1}}%
{\XXint\scriptstyle\scriptscriptstyle{#1}}%
{\XXint\scriptscriptstyle\scriptscriptstyle{#1}}%
\!\int}
\def\XXint#1#2#3{{\setbox0=\hbox{$#1{#2#3}{\int}$ }
\vcenter{\hbox{$#2#3$ }}\kern-.6\wd0}}

\def\dashint{\Xint-}

\begin{document}
\title[Interior nodal sets of Steklov eigenfunctions on surfaces]{Interior nodal sets of Steklov
eigenfunctions on surfaces}
\author{ Jiuyi Zhu}
\address{
 Department of Mathematics\\
Johns Hopkins University\\
Baltimore, MD 21218, USA\\
Emails: jzhu43@math.jhu.edu}
\thanks{\noindent  Research is partially
supported by the NSF grant DMS 1500468. }
\date{}
 \dedicatory{}

\begin{abstract}
We investigate the interior nodal sets $\mathcal{N}_\lambda$ of
Steklov eigenfunctions on connected and compact surfaces with
boundary. The optimal vanishing order of Steklov eigenfunctions  is
shown be $C\lambda$. The singular sets  $\mathcal{S}_\lambda$ are
finite points on the nodal sets. We are able to prove that the
Hausdorff measure $H^0(\mathcal{S}_\lambda)\leq C\lambda^2$.
Furthermore, we obtain an upper bound for the measure of interior
nodal sets $H^1(\mathcal{N}_\lambda)\leq C\lambda^{\frac{3}{2}}$.
Here those positive constants $C$ depend only on the surfaces.
\end{abstract}

\maketitle
\section{Introduction}
Let $(\mathcal{M}, g)$ be a smooth, connected and compact surface
with smooth boundary $\partial \mathcal{M}$. The
 main goal of this paper is to obtain an upper bound of interior nodal
sets
$$\mathcal{N}_\lambda=\{ z\in\mathcal{M}|e_\lambda=0\}$$
for Steklov eigenfunctions
\begin{equation}
\left\{
\begin{array}{lll}
\triangle_g e_\lambda=0,
 \quad z\in\mathcal{M},
\medskip \\
\frac{\partial e_\lambda}{\partial \nu}(z)=\lambda e_\lambda(z),
\quad z\in
\partial\mathcal{M},
\end{array}
\right. \label{Stek}
\end{equation}
where $\nu$ is a unit outward normal on $\partial\mathcal{M}$.
The Steklov eigenfunctions were introduced by Steklov in 1902 for
bounded domains in the plane. It interprets the steady state
temperature distribution in the domain such that the heat flux on
the  boundary is proportional to the temperature. It is also found
applications in a quite few physical fields, such as fluid
mechanics, electromagnetism, elasticity, etc.
 Especially, the model (\ref{Stek}) was studied by Calder\'{o}n
\cite{C} as it can be regarded as eigenfunctions of the
Dirichlet-to-Neumann map. The interior nodal sets of Steklov
eigenfunctions represent the stationary points in $\mathcal{M}$. In
the context of quantum mechanics, nodal sets are the sets where a
free particle is least likely to be found.

It is well-known that the spectrum $\lambda_j$ of Steklov eigenvalue
problem is discrete
 with $$0=\lambda_0<\lambda_1\leq
\lambda_2\leq \lambda_3,\cdots, \
 \ \mbox{and} \ \ \lim_{j\to\infty}\lambda_j=\infty.$$
There exists an orthonormal basis $\{ e_{\lambda_j}\}$ of
eigenfunctions such that
$$ e_{\lambda_j}\in C^\infty(\mathcal{M}), \quad
\int_{\partial\mathcal{M}} e_{\lambda_j} e_{\lambda_k}\, dV_{
{g}}=\delta_{j}^{k}.
$$
Estimating the Hausdorff measure of nodal sets has always been an
important subject concerning the study of eigenfunctions. This
subject centers around the famous Yau's conjecture. Recently, much
work has been devoted to the bounds of nodal sets
$$ Z_\lambda=\{z\in\partial\mathcal{M}| e_\lambda(z)=0\}      $$
of Steklov eigenfunctions on the  boundary.  Bellova and Lin
\cite{BL} proved the $H^{m-1}(Z_\lambda)\leq C\lambda^6$ with $C$
depending only on $\mathcal{M}$, if $\mathcal{M}$ is a $m+1$
dimensional analytic manifold. Zelditch \cite{Z1} improved their
results and gave the optimal upper bound $H^{m-1}(Z_\lambda)\leq
C\lambda$ for analytic manifolds using microlocal analysis. For the
smooth manifold $\mathcal{M}$, by assuming  that $0$ is a regular
value, Wang and the author in \cite{WZ} recently established a lower
bound
$$H^{m-1}(Z_\lambda)\geq C\lambda^{\frac{3-m}{2}}.$$

Before presenting our results for interior nodal sets, let's briefly
review the literature about the nodal sets of classical
eigenfunctions. Interested reader may refer to the book \cite{HL}
and survey \cite{Z} for detailed account about this subject. Let
$e_\lambda$ be $L^2$ normalized eigenfunctions of Laplacian-Beltrami
operator on compact manifolds $(\mathcal{M}, g)$ without boundary,
\begin{equation} -\triangle_g e_\lambda=\lambda^2 e_\lambda.
\label{class}\end{equation} Yau's conjecture states that for any
smooth manifold, one should control the upper and lower bound of
nodal sets of classical eigenfunctions as
\begin{equation}c\lambda\leq H^{n-1}(
\mathcal{N}_\lambda)\leq C\lambda \label{yau}
\end{equation} where $C, c$
depends only on the manifold $\mathcal{M}$. The conjecture is only
verified for real analytic manifolds by Donnelly-Fefferman in
\cite{DF}. Lin \cite{Lin} also showed the upper bound for the analytic manifolds by a different approach.  For the smooth manifolds, the conjecture is still not
settled. For the lower bound of nodal sets with $n\geq 3$, Colding
and Minicozzi \cite{CM}, Sogge and Zelditch \cite{SZ}, \cite{SZ1}
independently obtained that
$$ H^{n-1}(\mathcal{N}_\lambda)\geq C\lambda^{\frac{3-n}{2}}      $$
for smooth manifolds. See also \cite{HSo} for deriving the same
bound by adapting the idea in \cite{SZ}. For the upper bound, Hardt
and Simon \cite{HS} gave an exponential upper bound
$$ H^{n-1}(\mathcal{N}_\lambda)\leq C e^{\lambda \ln \lambda}.      $$
In surfaces, better results have been obtain.
 Br\"uning \cite{Br} and Yau (unpublished) derived the same lower bound
 as (\ref{yau}). The best estimate to date for the upper bound
 is
$$ H^{1}(\mathcal{N}_\lambda)\leq C\lambda^{\frac{3}{2}}         $$
by Donnelly-Fefferman \cite{DF2} and Dong \cite{D} using different
methods.

Let us return to Steklov eigenvalue problem (\ref{Stek}). By the
maximum principle, there exist nodal sets in the manifold
$\mathcal{M}$ and those sets must intersect the boundary
$\partial\mathcal{M}$. Thus, it is natural to study the size of
interior nodal sets in $\mathcal{M}$. We can also ask Yau's type
questions about the Hausdorff measure of nodal sets. The natural and
corresponding conjecture for Steklov eigenfunctions should states
exactly the same as (\ref{yau}). See also the open questions in the
survey by Girouard and Polterovich in \cite{GP}. Recently, Sogge,
Wang and the author \cite{SWZ} obtained a lower bound for interior
nodal sets
$$H^{n-1}({\mathcal{N}}_\lambda)\geq
C\lambda^{\frac{2-n}{2}}$$ for $n$-dimensional manifold
$\mathcal{M}$. Very recently, Polterovich, Sher and Toth \cite{PST}
can verify Yau's type conjecture for (\ref{Stek}) on real-analytic
Riemannian surfaces.

An interesting topic related to the measure of nodal sets is about
doubling inequality. Based on doubling inequalities, one can obtain
the vanishing order of eigenfunctions, which characterizes how fast
the eigenfunctions vanish.
  For the
classical eigenfunctions of (\ref{class}), Donnelly-Fefferman in \cite{DF},
\cite{DF1} obtained that the maximal vanishing order of $e_\lambda$
is at most order $C\lambda$ everywhere. To achieve it, a doubling
inequality
\begin{equation}
\int_{\mathbb B (z_0,\, 2r)}e_\lambda^2\leq Ce^{\lambda}
\int_{\mathbb B(z_0,\, r)}e_\lambda^2 \label{doub}
\end{equation}
is derived using Carleman estimates, where $\mathbb B(p, c)$ denotes
as a ball centered at $p$ with radius $c$. The doubling estimate
(\ref{doub}) plays an important role in obtaining the bounds of
nodal sets for analytic manifolds in \cite{DF} and upper bound of
nodal sets for smooth surfaces in \cite{DF2}. For the Steklov
eigenfunctions, we have obtained a doubling inequality on the
boundary $\partial\mathcal{M}$ and derive that the sharp vanishing
order is less than $C\lambda$ on the boundary $\partial\mathcal{M}$.
For steklov eigenfunction in $\mathcal{M}$, we are also able to get
the doubling inequality, see proposition \ref{proo}. With aid of
doubling estimates and Carleman inequalities, The following optimal
vanishing order for Steklov eigenfunctions can be obtained.

\begin{theorem}
The vanishing order of Steklov eigenfunction $e_\lambda$ of
(\ref{Stek}) in $\mathcal{M}$ is everywhere less than $C\lambda$.
\label{th3}
\end{theorem}
 It's sharpness can be seen  in the case that the manifold $
\mathcal{M}$ is a ball. Notice that the doubling estimates in
proposition \ref{proo} and the vanishing order in Theorem \ref{th3}
hold for any $n$-dimensional compact manifolds.

Singular sets
$$\mathcal{S}_\lambda=\{ z\in\mathcal{M}|e_\lambda=0, \nabla e_\lambda=0\}
$$ are contained in nodal sets.
In Riemannian surfaces, those singular sets are finite points in the
1-dimensional nodal sets. It is interesting to count the number of
those singular sets. Based on a Carleman inequality with
singularities, we are able to show an upper bound of singular sets.
\begin{theorem}
Let $(\mathcal{M}, g)$ be a smooth, compact surface with
smooth boundary $\partial\mathcal{M}$. There holds
\begin{equation}H^0(\mathcal{S}_\lambda)\leq C\lambda^2
\end{equation} for Steklov eigenfunctions in (\ref{Stek}). \label{th2}
\end{theorem}

For the nodal sets of Steklov eigenfunctions, we are able to build a
similar type of Carleman inequality as \cite{DF2} and show the
following result.
\begin{theorem}
Let $(\mathcal{M}, g)$ be a smooth, compact surface with
smooth boundary $\partial\mathcal{M}$. Then
\begin{equation}
H^1(\mathcal{N}_\lambda)\leq C\lambda^{\frac{3}{2}}
\end{equation}
holds for Steklov eigenfunctions in (\ref{Stek}). \label{th1}
\end{theorem}

The outline of the paper is as follows. Section 2 is devoted to
reducing the Steklov eigenvalue problem into an equivalent elliptic
equation without boundary. Then we obtain the optimal doubling
inequality and show theorem \ref{th3}. In section 3, we establish
the Carleman inequality with singularities at finite points. Under
additional assumptions of those singular points, a stronger Carleman
inequality is derived. We measure the singular sets in section 4.
Section 5, 6 and 7 are devoted to obtaining the nodal length of
Steklov eigenfunctions. Under the slow growth of $L^2$ norm
condition, we find out the nodal length in section 6. Based on a
similar type of Calder\'on and Zygmund decomposition procedure, we
show the slow growth at almost every point. Then the measure of
nodal sets is arrived by summing up the nodal length in each small
square. The letter $c$, $C$, $C_i$, $d_i$ denote generic positive
constants and do not depend on $\lambda$. They may vary in different
lines and sections.

\noindent{\bf Acknowledgement.} It is my pleasure to thank Professor
Christopher D. Sogge for helpful discussions about this topic and
guidance into the area of eigenfunctions. I also would like to thank
X. Wang for many fruitful conversations.

\section{Vanishing Order of Steklov Eigenfunctions}
In this section, we will reduce the Steklov eigenvalue problem to an
equivalent model on a boundaryless manifold. The presence of
eigenvalue on the boundary $\partial\mathcal{M}$ will be reflected
on the coefficient functions of a second order elliptic equation.
Let $d(z)= dist\{z,
\partial\mathcal{M}\}$ denote the geodesic distance function from $x\in\mathcal{M}$ to the
boundary $\partial\mathcal{M}$. Since $\mathcal{M}$ is smooth, there
exist a $\rho$-neighborhood of $\partial\mathcal{M}$ in
$\mathcal{M}$ such that $d(x)$ is smooth in the neighborhood. Let's
denoted it as $\mathcal{M}_{\rho}$. We extend $d(z)$ smoothly in
$\mathcal{M}$ by
\begin{equation}
\delta(z)=\left \{\begin{array}{lll} d(z) \quad z\in \mathcal{M}_{\rho}, \medskip\nonumber\\
l(z) \quad z\in \mathcal{M}\backslash\mathcal{M}_{\rho},
\end{array}
\right.
\end{equation}
where $l(z)$ is a smooth function in
$\mathcal{M}\backslash\mathcal{M}_{\rho}$. Note that the extended
function $\delta(z)$ is a smooth function in $\mathcal{M}$. We first
reduce Steklov eigenvalue problem into an elliptic equation with
Neumann boundary condition. Let
$$v(z)=e_\lambda \exp\{\lambda \delta(z)\}.  $$
It is known that $v(z)=e_\lambda(z)$ on $\partial \mathcal{M}$.  For
$z\in
\partial \mathcal{M}$, $\nabla_g \delta(z)=-\nu(z)$. Recall that $\nu(z)$ is the unit outer normal on $z\in\partial\mathcal{M}
$. We can check that the new function $v(z)$ satisfies
\begin{equation}
\left \{ \begin{array}{lll} \triangle_g v+b(z)\cdot\nabla_g
v+q(z)v=0 &\quad \mbox{in}\ \mathcal{M},
\medskip\\
\frac{\partial v}{\partial \nu}=0 &\quad \mbox{on}\
\partial\mathcal{M},
\end{array}
\right.
\end{equation}
with
\begin{equation}
\left \{\begin{array}{lll}
b(z)=-2\lambda \nabla_g \delta(z),   \medskip\\
q(z)=\lambda^2|\nabla_g \delta(z)|^2-\lambda\triangle_g \delta(z).
\end{array}
\right. \label{fffw}
\end{equation}
In order to get rid of boundary condition, we attach two copies of
$\mathcal{M}$ along the boundary and consider a double manifold
$\overline{\mathcal{M}}=\mathcal{M}\cup \mathcal{M}$. The metric $g$
extends to $\overline{\mathcal{M}}$ with Lipschitz type singularity
along $\partial\mathcal{M}$, since the lift metric $ {g'}$ of $g$ on
$\mathcal{M}$ to the double manifold $\overline{\mathcal{M}}$ is
Lipschitz. There also exists a canonical involutive isometry
$\mathcal {F}: \overline{\mathcal{M}}\to \overline{\mathcal{M}}$
that interchanges the two copies of ${\mathcal{M}}$. Then the
function $v(x)$ can be extended to $\overline{\mathcal{M}}$ by
$v\circ \mathcal {F}= v$. Therefore, $v(z)$ satisfies
\begin{equation}
\triangle_{g'} v+\bar{b}(z)\cdot\nabla_{g'} v+\bar{ q} (z)v=0 \quad
\mbox{in}\ \overline{\mathcal{M}}.  \label{star}
\end{equation}
From (\ref{fffw}), one can see that
\begin{equation}
\left \{\begin{array}{lll}
\|\bar b\|_{W^{1, \infty}(\overline{\mathcal{M}})}\leq C\lambda, \medskip\\
\|\bar q\|_{W^{1, \infty}(\overline{\mathcal{M}})}\leq C\lambda^2.
\label{core}
\end{array}
\right.
\end{equation}
After this procedure, we can instead study the nodal sets for the
 second order elliptic equation (\ref{star}) with assumption
(\ref{core}). Note that $\overline{\mathcal{M}}$ is a manifold
without boundary.

We present a brief proof of Theorem \ref{th3}. It is a small
modification of the argument in \cite{Zh}, where the sharp vanishing
order of Steklov eigenfunctions on boundary $\partial\mathcal{M}$ is
shown to be less than $C\lambda$.
\begin{proof}[Proof of Theorem \ref{th3}]
 By a standard regularity
argument, we can still consider polar coordinate for (\ref{star}).
We are able to establish a similar Carleman inequality as \cite{Zh}
 for the general second order elliptic
equation (\ref{star}). See also e.g. \cite{BC}.

\begin{lemma}
Let $u\in C^{\infty}_{0}(\frac{1}{2}\eps_1<r<\eps_0)$. If
$\tau>C_1(1+\|\bar b\|_{W^{1, \infty}}+\|\bar q\|^{1/2}_{W^{1,
\infty}})$. Then
\begin{equation}
\int r^4 e^{2\tau \phi(r)}|\triangle_{g'} u+\bar b \cdot \nabla_{g'}
u+ \bar q u|^2\, drd\omega\geq C_2 \tau^3\int r^{\eps}
e^{2\tau\phi(r)} u^2 \,drd\omega, \label{cca}
\end{equation}
where $\phi(r)=-\ln r+r^\eps$ and $r$ is the geodesic distance.
$0<\eps_0, \eps_1, \eps<1$ are some fixed constants. Moreover, $(r,
\omega)$ are the standard polar coordinates. \label{carl}
\end{lemma}

Using this Carleman estimate and choosing suitable test functions, a
Hadamard's three-ball result can be obtained in
$\overline{\mathcal{M}}$. There exist constants $r_0,$ $C$ and
$0<\gamma<1$ depending only on $\overline{\mathcal{M}}$ such that
for any solutions of (\ref{star}), $0<r<r_0$, and
$z_0\in\overline{\mathcal{M}}$, one has
\begin{equation}
\int_{\mathbb B(z_0,\, r)}{ v}^2\leq e^{C(1+\|\bar b\|_{W^{1,
\infty}}+{\|\bar q\|}^{1/2}_{W^{1, \infty}})}(\int_{\mathbb B(z_0,\,
{2r})}{v}^2)^{1-\gamma}(\int_{\mathbb B(z_0,\, {r/2})}{
v}^2)^{\gamma}.
\end{equation}
Based on a propagation of smallness argument using the three-ball
result and Carleman estimates (\ref{cca}), taking the assumptions
(\ref{core}) into account, we are able to obtain the doubling
inequality in $\overline{\mathcal{M}}$.
\begin{proposition}
There exist constants $r_0$ and $C$  depending only on
$\overline{\mathcal{M}}$ such that for any $0<r<r_0$ and $z_0\in
\overline{\mathcal{M}}$, there holds
\begin{equation}
\|v\|_{L^2(\mathbb B(z_0, \,{2r}))}\leq
e^{C\lambda}\|v\|_{L^2(\mathbb B(z_0,\, {r}))} \label{ddou}
\end{equation}
for any solutions of (\ref{star}). \label{proo}
\end{proposition}
One can see that the doubling estimate holds in $\mathcal{M}$ if
$\mathbb B(z_0, {2r})\subset\mathcal{M}$. By standard elliptic
estimates, one can have $L^\infty$ norm of doubling inequality
$$
\|v\|_{L^\infty(\mathbb B(z_0, \,{2r}))}\leq
e^{C\lambda}\|v\|_{L^\infty(\mathbb B(z_0,\, {r}))}.
$$
Since $\overline{\mathcal{M}}$ is compact, we can derive that
$$ \|v\|_{L^\infty(\mathbb B(z_0, \,{r}))} \geq r^{C\lambda} $$
for any $z_0\in \overline{\mathcal{M}}$, which implies the vanishing
order for $v$ is less than $C\lambda$. So is the vanishing order of
$u$. This completes Theorem \ref{th3}.
\end{proof}
\section{Carleman estimates}

This section is to devoted to establishing Carleman inequalities
involving weighted functions at finite points. We will consider the
behavior of $v$ in a conformal coordinate patch. Since
$\overline{\mathcal{M}}$ is a compact Riemannian surface. There
exists a finite number $N$ of conformal charts $( \mathcal {U}_i,
\phi_i)$ with $\phi_i: \mathcal {U}_i\subset
\overline{\mathcal{M}}\to \mathcal{V}_i\subset \mathbb R^2$ and
$i\in\{1, 2, \cdots N\}$. On each of these charts, the metric is
conformally flat and there exists positive function $g_i$ such that
$g'=g_i(x, y)(dx^2+dy^2)$. By the compactness of the surface, there
is positive constants $c$ and $C$ such that $0<c<g_i<C$ for
each $i$. Under this equivalent metric, the equation (\ref{star})
can be written as
\begin{equation}
\triangle v+{{\bar{b}}}(z)\cdot\nabla v+{ {\bar{q}}} (z)v=0 \quad
\mbox{in}\ \mathcal {V}_i, \label{star1}
\end{equation}
where $ \triangle$ is Euclidian Laplacian, $\nabla$ is  Euclidian
gradient and $z=(x,\,y)$. We use the same notations $\bar{b}(z)$ and
$\bar{q}(z)$ as that in (\ref{star1}), since they satisfy the same
conditions as (\ref{core}). They only differ by some function about
$g_i$.

By restricting into a small ball $\mathbb B(p, 3c)$ contained in the
conformal chart, we consider $v$ in the small ball. Let $\bar
v(z)=v(cz)$. It follows from (\ref{star1}) that
\begin{equation}
\triangle \bar v+{\tilde{{b}}}(z)\cdot\nabla \bar v+{ \tilde{{q}}}
(z)\bar v=0 \quad \mbox{in}\ \mathbb {B}_3, \label{star2}
\end{equation}
with $\tilde{{b}}=c\bar b$ and $\tilde{q}=c^2 \bar q$. If $c$ is
sufficiently small, $\tilde{{b}}$ and $\tilde{q}$ are arbitrary
small.

The crucial tool in \cite{DF2} is a Carleman inequality for
classical eigenfunctions involving weighted functions with
singularity at finite points. We will obtain the corresponding
Carleman inequality for the second order  elliptic equation
(\ref{star2}). We adapt the approach in \cite{DF2} to obtain the
desirable Carleman estimate for (\ref{star2}).

Let $\mathcal {D}\subset \mathcal {C}$ be an open set and $\psi\in
C^\infty_{0}(\mathcal {D})$ be a real valued unction. We introduce
the following differential operators
$$\mathcal{\partial} =\frac{1}{2}( \frac{\partial}{\partial x}-i \frac{\partial}{\partial
y}), \quad \quad \overline{\partial} =\frac{1}{2}(
\frac{\partial}{\partial x}+i \frac{\partial}{\partial y}).$$ Direct
computation shows that
$\overline{\partial}\partial\psi=\frac{1}{4}\triangle \psi$. By the
Cauchy-Riemann equation, $u$ is holomorphic if and only if
$\overline{\partial}u=0$. For the completeness, we present the
elementary inequality in
 \cite{DF2}.

 \begin{lemma}
let $\Phi$ be a smooth positive function in $\mathcal {D}$. Then
\begin{equation}
\int_{\mathcal {D}}|\overline{\partial}u|^2\Phi\geq
\frac{1}{4}\int_{\mathcal {D}}(\triangle \ln \Phi)|u|^2\Phi.
\end{equation}
Here the integral is taken with respect to the lebesgue measure.
\label{lem1}
 \end{lemma}
We want  the weight function to involve those singular points. To
specialize the choice of $\Phi $, we construct the following
function $\psi_0$.
\begin{lemma}
There exists a smooth function $\psi_0$ defined for $|z|>1-2a$
satisfying the following properties: \\
(i) $a_1\leq \psi_0(z)\leq a_2$ with constant $a_1, a_2>0$.\medskip\\
(ii) $\psi_0=1$ on $\{|z|>1\}$.\medskip \\
(iii) $\triangle \ln \psi_0\geq 0$ on $\{|z|>(1-2a)\}$.\medskip\\
(iv) If $1-2a<|z|<1-a$, then $\triangle \ln \psi_0\geq a_3>0$.
\end{lemma}
The existence of such $\psi_0$ follows from existence and unique
theory of ordinary differential equations.

We assume that
 $$D_l=\{z|\, |z-z_l|\leq \delta\}.$$
 Let $D_l$ be a finite collection of pairwise
disjoint disks, which are contained in a unit disk centered at
origin. Let $$D_l(a)=\{z|\, |z-z_l|\leq (1-2a)\delta\}$$ be the
smaller concentric disk. We define a smooth weight function
$\Psi_0(z)$ as
\begin{equation}
\Psi_0(z)=\left \{ \begin{array}{lll} 1 \quad \mbox{if} \ z\not\in
\cup_l
D_l,\nonumber \medskip \\
\psi_0(\frac{z-z_l}{\delta}) \quad \mbox{if} \ z\in D_l. \nonumber
\end{array}
\right.
\end{equation}
We also introduce the following domain
$$ A_l=\{ (1-2a)\delta\leq |z-z_l|\leq (1-a)\delta\}.$$
 From the last lemma, $\Psi_0(z)$ satisfies those properties:\\
(i) $a_1\leq \Psi_0(z)\leq a_2$. \medskip\\
(ii) $\triangle \ln \Psi_0\geq 0$ for $z\in \mathbb R^2\backslash
\cup_l D_l(a)$. \medskip \\
(iii) $\triangle \ln \Psi_0\geq a_3 \delta^{-2}$ for $z\in A_l$. \\
Note that $a_i$ in above are positive constants independent of
$\lambda$. Denote $$A=\cup_l A_l.$$ Suppose that $\tau$ is a
nonnegative constant. We introduce $\Phi(z)=\Psi_0(z)e^{\tau
|z|^2}$. For $u\in C^{\infty}_0(\mathbb R^2\backslash \cup_l
D_l(a))$, we assume that $\mathcal {D}$ contains the support of $u$
and $A\subset \mathcal {D}\subset \mathbb R^2\backslash \cup_l
D_l(a)$. Obviously,
$$ \ln \Phi(z)=  \ln\Psi_0(z) +\tau |z|^2.        $$
Substituting $\Phi$ in lemma \ref{lem1} gives that
\begin{equation}\int_{\mathcal{D}}|\overline{\partial} u|^2\Psi_0(z)e^{\tau
|z|^2}\geq C_1\tau \int_{\mathcal{D}}|u|^2\Psi_0(z)e^{\tau
|z|^2}+C_2 \delta^{-2} \int_{A}|u|^2e^{\tau |z|^2},
\end{equation}
where we have used the properties (ii) and (iii) for $\Psi_0$. The
boundedness of $\Psi_0(z)$ yields that
\begin{equation}
\int_{\mathcal{D}}|\overline{\partial} u|^2 e^{\tau |z|^2}\geq
C_3\tau \int_{\mathcal{D}}|u|^2e^{\tau |z|^2}+C_4 \delta^{-2}
\int_{A}|u|^2e^{\tau |z|^2}. \label{impt}
\end{equation}
Define a holomorphic function $$P(z)=\prod_l (z-z_l).$$ Then
$\overline{\partial}({u}/{P})={\overline{\partial} u}/P$. Replacing
$u$ by ${u}/{P}$ in (\ref{impt}), it follows that
\begin{equation}
\int_{\mathcal{D}}|\overline{\partial} u|^2 |P|^{-2}e^{\tau
|z|^2}\geq C_3\tau \int_{\mathcal{D}}|u|^2|P|^{-2}e^{\tau |z|^2}+C_4
\delta^{-2} \int_{A}|u|^2|P|^{-2}e^{\tau |z|^2}. \label{impt1}
\end{equation}

We will establish a Carleman inequality for second order elliptic
equations as (\ref{star2}). Write $\tilde{b}(x)= (\tilde{b}_1(x),\,
\tilde{b}_2(x)).$  Let
$$u=\partial f+\frac{1}{2}(\tilde{b}_1- i\tilde{b}_2)f, $$
where $f\in C_0^\infty(\mathbb R^2\backslash \cup_l D_l(a))$ is a
real valued function. Then
$$\overline{\partial} u=\frac{1}{4}[\triangle f+div\tilde{b} f+\tilde{b}\cdot\nabla f+i(\frac{\partial(\tilde{b}_1f)}{\partial y}
-\frac{\partial(\tilde{b}_2f)}{\partial x})].      $$ Plugging above
$u$ into (\ref{impt1}), we obtain
\begin{eqnarray}
\int_{\mathcal{D}}\big[|\triangle f &+&\tilde{b}\cdot\nabla
f|^2+|div\tilde{b} f|^2 +|\frac{\partial(\tilde{b}_1f)}{\partial y}
-\frac{\partial(\tilde{b}_2f)}{\partial x}|^2\big]|P|^{-2} e^{\tau
|z|^2}  \nonumber\\
&\geq & C_3\tau\int_{\mathcal{D}}|\nabla f|^2 |P|^{-2}e^{\tau
|z|^2}-
C_3\tau\int_{\mathcal{D}}|\tilde{b}|^2|f|^2 |P|^{-2}e^{\tau |z|^2} \nonumber\\
&+ &C_4\delta^{-2}\int_{A}|\nabla f|^2|P|^{-2}e^{\tau |z|^2}-
C_4\delta^{-2}\int_{A}|\tilde{b}|^2|f|^2|P|^{-2}e^{\tau |z|^2}.
\label{many}
\end{eqnarray}
If we choose $u=f$ in (\ref{impt1}), we get
\begin{equation}
\int_{\mathcal{D}}|\nabla f|^2 |P|^{-2}e^{\tau |z|^2}\geq C_3\tau
\int_{\mathcal{D}}|f|^2|P|^{-2}e^{\tau |z|^2}. \label{impt2}
\end{equation}
Since the norm of $\tilde{b}$ is chosen small enough, it is smaller
than $\tau$ which will be chosen large enough. With aid of
(\ref{impt2}), we can incorporate the terms about $\tilde{b}$ in the
left hand side of (\ref{many}) into the first term in the right hand
side of (\ref{many}),
\begin{eqnarray}
\int_{\mathcal{D}}|\triangle f +\tilde{b}\cdot\nabla f|^2 |P|^{-2}
e^{\tau |z|^2}& \geq & C_5 \tau \int_{\mathcal{D}}|\nabla
f|^2|P|^{-2} e^{\tau |z|^2} +C_4\delta^{-2}\int_{A}|\nabla
f|^2|P|^{-2}e^{\tau |z|^2} \nonumber\\&-&
C_4\delta^{-2}\int_{A}|\tilde b|^2|f|^2|P|^{-2}e^{\tau |z|^2}.
\label{hao}
\end{eqnarray}
Furthermore, if $u=f$, the inequality (\ref{impt1}) implies that
\begin{equation} \int_{\mathcal{D}}|\nabla
f|^2 |P|^{-2}e^{\tau |z|^2}\geq C_4 \delta^{-2}
\int_{A}|f|^2|P|^{-2}e^{\tau |z|^2}. \label{any}
\end{equation}
Applying (\ref{any}) to the last term in the right hand side of
(\ref{hao}) gives that
\begin{eqnarray}
\int_{\mathcal{D}}|\triangle f +\tilde{b}\cdot\nabla f|^2 |P|^{-2}
e^{\tau |z|^2}& \geq & C_6 \tau^2 \int_{\mathcal{D}}|
f|^2|P|^{-2} e^{\tau |z|^2} \nonumber \\
&+&C_7\delta^{-2}\int_{A}|\nabla f|^2|P|^{-2}e^{\tau |z|^2}.
\label{keys}
\end{eqnarray}

We continue to get a refined estimate for the last term of
(\ref{keys}). In order to achieve this goal,  we need the following
hypotheses
for the geometry of the disk $D_l$ and the parameter $\tau>1$: \\
($R_{1}$) The radius $\delta$ for each disk $D_l$ is less than $a_4
\tau^{-1}$. \\
($R_{2}$) The distance between any two distinct $z_l$ is at least
$2a_5\tau^{\frac{1}{2}} \delta$.\\
($R_{3}$) The total number of disk $D_l$ is at most $a_6\tau$. \\
 Under the those assumptions, we have those comparison
estimates from \cite{DF2}.

\begin{lemma}
If $\bar z_1$ and $\bar z_2$ are any points in the same component
$A_l$ of $A$, then \\
(i) $a_7< e^{\tau |\bar z_1|^2}/e^{\tau |\bar z_2|^2}<a_8$.\medskip\\
(ii) $a_9<|P(\bar z_1)|/|P(\bar z_2)|<a_{10}$. \label{lemc}
\end{lemma}

We also need the following Poincar\'e type inequality on each
annulus. If $f\in C^{\infty}(A_l)$ and $f$ vanishes on the inner
boundary of $A_l$, then
\begin{equation}
\int_{A_l} |\nabla f|^2 \geq a_{11} \delta^{-2}\int_{A_l}|f|^2.
\label{poin}
\end{equation}
The proof of (\ref{poin}) can be seen in \cite{DF2}. Let $z_l\in
A_l$ be chosen arbitrarily. By lemma \ref{lemc}, it follows that
$$ \int_{A_l}|\nabla f|^2|P(z)|^{-2}e^{\tau |z|^2} \geq  C_8\sum_{l}e^{\tau |z_l|^2}|P(z_l)|^{-2}\int_{A_l} |\nabla f|^2.$$
Since $f\in C_0^\infty(\mathbb R^2\backslash \cup_l D_l(a))$, the
inequality (\ref{poin}) yields that
$$ \int_{A_l}|\nabla f|^2|P(z)|^{-2}e^{\tau |z|^2} \geq
C_9\sum_{l}e^{\tau |z_l|^2}|P(z_l)|^{-2}\delta^{-2}\int_{A_l} |
f|^2.$$ Using lemma \ref{lemc} again, we obtain
$$ \int_{A_l}|\nabla f|^2|P(z)|^{-2}e^{\tau |z|^2} \geq C_{10}
\delta^{-2}\int_{A}| f|^2|P(z)|^{-2}e^{\tau |z|^2}.$$ Substituting
the last inequality into the last term in (\ref{keys}) leads to
\begin{eqnarray}
\int_{\mathcal{D}}|\triangle f +\tilde{b}\cdot\nabla f|^2 |P|^{-2}
e^{\tau |z|^2}& \geq & C_6 \tau^2 \int_{\mathcal{D}}|
f|^2|P|^{-2} e^{\tau |z|^2} \nonumber \\
&+&C_{11}\delta^{-4}\int_{A}| f|^2|P|^{-2}e^{\tau |z|^2}.
\label{keys1}
\end{eqnarray}
We summarize the above arguments in the following proposition.
\begin{proposition}
Assuming $f\in C_0^\infty(\mathbb R^2\backslash \cup_l D_l(a))$.
Then \\
(i) it holds that
\begin{equation}
\int_{\mathcal{D}}|\triangle f +\tilde{b}\cdot\nabla f|^2 |P|^{-2}
e^{\tau |z|^2} \geq  C \tau^2 \int_{\mathcal{D}}| f|^2|P|^{-2}
e^{\tau |z|^2}. \label{rec}
\end{equation}
(ii) If the additional assumptions $(R_1)$-$(R_3)$ for $D_l$ hold, the
stronger inequality (\ref{keys1}) is satisfied. \label{pro21}
\end{proposition}

\section{Measure of Singular Sets}

Let $\mathcal {M}$ be a compact smooth surface. In section 2, we
have shown that the Steklov eigenfunction $e_\lambda$ vanishes at
any points at most $C\lambda$. By the implicit function theorem,
outside the singular sets, the nodal set is locally a 1-dimensional
$C^1$ manifold. Adapting the arguments in \cite{DF2} for
(\ref{star2}), we can estimate those singular points in a
quantitative way. We are able to obtain an upper bound for the
singular points in term of eigenvalue $\lambda$.

\begin{lemma}
Singular sets consist of at most finite many points.
\end{lemma}
\begin{proof}
Without loss of generality, we assume that $0\in \mathcal
{S}_\lambda$ and choose normal coordinate $(x, y)$ at the origin.
 Next we prove there are finite singular points
in $\overline{\mathcal {M}}$. Using Taylor expansion, we expand $v$ locally at
origin. Then $v(x, y)=F_j(x, y)+W_{j+1}(x, y)$, where $F_j(x,y)$
consists of the leading nonvanishing term with homogenous order
$j\geq 2$. $W_{j+1}(x,y)$ is a higher order reminder term. Since
$\triangle v+{{\bar{b}}}(z)\cdot\nabla v+{ {\bar{q}}} (z)v=0$ and
the coordinate is normal, we obtain that $\triangle F_j=0$. Under
polar coordinates, we find that $F_j=r^j\big(a_1
\cos(j\theta)+a_2\sin(j\theta)\big)$. Obviously,
$r^{-1}\frac{\partial F_j}{\partial \theta}$ and $\frac{\partial
F_j}{\partial r}$ have no common zero if $r\not=0$. Since
$$|\nabla F_j|^2=|\frac{\partial F_j}{\partial r}|^2 +\frac{1}{r^2}|\frac{\partial F_j}{\partial \theta}|^2,     $$
there exists a small neighborhood of $\mathcal {U}$ of origin such
that $\mathcal {U}\cap \mathcal {S}_\lambda=0 $. Since
$\overline{\mathcal {M}}$ is compact, then the lemma follows.
\end{proof}

We plan to count the number of singular points in a sufficiently
small ball. Let $p\in \overline{\mathcal {M}}$. Consider a geodesic
ball $\mathbb B(p, c\lambda^{-\frac{1}{2}})$. If $c$ is small
enough, then this geodesic ball is contained in a conformal chart.
 If
we choose
$$ w(z)= v(c {\lambda^{-\frac{1}{2}}z})$$
with $c$ sufficiently small. From equation (\ref{star1}), $w$
satisfies
\begin{equation}
\triangle w+\hat{{b}}(x)\cdot\nabla w+\hat{{q}}(x)w=0 \quad
\mbox{in}\ \mathbb B( 0, 4), \label{starr}
\end{equation}
with $ \hat{b}(x)= c\lambda^{-\frac{1}{2}}\bar{b}(x)$ and
$\hat{q}(x)= c^2\lambda^{-1}{ {\bar{q}}}(x).$ From (\ref{core}), we
obtain
\begin{equation}
\left \{\begin{array}{lll}
\|\hat{b}\|_{W^{1, \infty}(\mathbb B( 0, 4))}\leq c\lambda^{\frac{1}{2}}, \medskip\\
\| \hat{q}\|_{W^{1, \infty}(\mathbb B( 0, 4))}\leq c^2\lambda
\label{core2}
\end{array}
\right.
\end{equation}
with $c$ sufficiently small.

Next we will count the total order of vanishing of singular points
for $w$ in the sufficiently small ball. We study $w$ in equation
(\ref{starr}),
\begin{proposition}
Suppose $z_l\in \mathcal{ S}_\lambda\cap\mathbb B(p, c\lambda^{-\frac{1}{2}})$
where $v$ vanishes to order $n_l+1$. Then $\sum_l n_l\leq C\lambda$.
\label{pro2}
\end{proposition}
\begin{proof}
It suffices to count the number of singular point of $w$ in a small
Euclidean ball with radius $\frac{1}{10}$ centered at origin.
Suppose that $w$ vanishes to order $n_l+1$. Let $n_l=m_l+1$. We
first consider the case $n_l\geq 2$. Then $m_l\geq 1$. Define the
polynomial
$$P(z)=\prod(z-z_l)^{m_l}$$ with $|z_l|<\frac{1}{10}$.  Let $\mathcal{D}=\mathbb B(0, \,2)$ and $D_l$ be small
disjoint disks of radius $\delta$ centered at $z_l$. If $f\in
C_0^\infty(\mathbb R^2\backslash \cup_l D_l)$, the inequality
(\ref{rec}) in proposition \ref{pro21} implies that
\begin{equation}
\int_{\mathcal{D}}(|\triangle f|^2+|\hat{b}\cdot\nabla f|^2)
|P|^{-2} e^{d_1\lambda |z|^2} \geq  C_2 \lambda^2\int_{\mathcal{D}}|
f|^2|P|^{-2} e^{d_1\lambda |z|^2}, \label{rec1}
\end{equation}
where $\tau=d_1\lambda$. We choose a cut-off function $\theta(z)$ such
that $\theta w$ is compact support in $\mathcal{D}$. We select the cut-off
function $\theta\in C^{\infty}_{0}(\mathcal{D}\backslash \cup_l D_l
) $ with following
properties: \\
(i) $\theta(z)=1$, \ if $|z|<\frac{3}{2}$ and $|z-z_l|>2\delta$. \medskip\\
(ii) $|\nabla \theta|<C_3$, $|\triangle\theta|<C_4$ if
$|z|>\frac{3}{2}$. \medskip \\
(iii) $|\nabla \theta|<C_5 \delta^{-1}$,
$|\triangle\theta|<C_6\delta^{-2}$ if $|z-z_l|<2\delta$.

Substituting $f=\theta w$ into (\ref{rec1}) yields that
\begin{eqnarray}
\int_{(|z|<\frac{3}{2})\cup(\frac{3}{2}\leq |z|\leq 2 )}|\triangle
(\theta w)+\tilde{b}\cdot\nabla (\theta w)|^2 |P|^{-2} e^{d_1\lambda
|z|^2} \geq C_2 \lambda^2\int_{|z|<\frac{3}{2}}| w|^2|P|^{-2}
e^{d_1\lambda |z|^2}. \nonumber
\end{eqnarray}
From the equation
(\ref{starr}),
$$\triangle (\theta w) +\hat{b}\cdot\nabla (\theta
w)=-\hat{q}\theta w+\triangle\theta w+2\nabla \theta\cdot\nabla
w+\hat{b}\cdot\nabla \theta w.
$$
By the assumption of $\theta$,
we obtain
 $$|\triangle \theta w|+|\nabla \theta\cdot\nabla w|+|\nabla \theta w|\leq C_7 \delta^{m_l},   \
 \mbox{if}  \ |z-z_l|\leq
2\delta.$$
Taking $\delta\to 0$,  by dominated convergence theorem, we have
\begin{eqnarray}
c\lambda^2\int_{|z|<\frac{3}{2}} | w|^2 |P|^{-2} e^{d_1\lambda
|z|^2} &+&C_9(1+\lambda)^2\int_{\frac{3}{2}\leq |z|\leq 2} (|
w|^2+|\nabla w|^2) |P|^{-2} e^{d_1\lambda |z|^2} \nonumber \\&& \geq
C_{10} \lambda^2\int_{|z|<\frac{3}{2}}| w|^2|P|^{-2} e^{d_1\lambda
|z|^2}. \label{revi}
\end{eqnarray}
Since $c$ is sufficiently small, we can absorb the first term in the
left hand side of (\ref{revi}) into the right hand side. Then \begin{equation}
\int_{\frac{3}{2}\leq |z|\leq 2} (| w|^2+|\nabla w|^2) |P|^{-2}
e^{d_1\lambda |z|^2}  \geq C_{11}\int_{|z|\leq \frac{1}{2}}|
w|^2|P|^{-2} e^{d_1\lambda |z|^2}.
\end{equation}
Obviously, it follows that
\begin{equation}
\max_{|z|\geq \frac{3}{2}} |P|^{-2}\int_{\frac{3}{2}\leq |z|\leq 2}
(| w|^2+|\nabla w|^2) e^{d_1\lambda |z|^2}  \geq C_{11}
(\min_{|z|\leq \frac{1}{2}}|P|^{-2})\int_{|z|\leq \frac{1}{2}}| w|^2.
\end{equation}
By standard elliptic theory, the last inequality implies

\begin{equation}
\max_{|z|\geq \frac{3}{2}} |P|^{-2}e^{d_2\lambda} \int_{|z|\leq
\frac{5}{2}} | w|^2\geq C_{11} (\min_{|z|\leq
\frac{1}{2}}|P|^{-2})\int_{|z|\leq \frac{1}{2}}| w|^2.
\label{than}
\end{equation}
We claim that \begin{equation}  e^{d_3 \sum m_l}\leq
\frac{\min_{|z|\leq \frac{1}{2}} |P|^{-2}} {\max_{|z|\geq
\frac{3}{2}} |P|^{-2}}. \label{more}
  \end{equation}
To prove (\ref{more}), it reduces to verify
\begin{equation}
 e^{-d_4 \sum m_l}\min_{|z|\geq \frac{3}{2}} |P|\geq \max_{|z|\leq \frac{1}{2}}
 |P|.
 \label{mmore}
\end{equation}
away from singular point $z_l$. Clearly,
$$ \max_{|z|\leq \frac{1}{2}} |P|\leq (\frac{1}{2})^{\sum m_l }. $$
Since $z_l\in \mathbb B (0, \frac{1}{10})$, we have
$$ (\frac{3}{4})^{\sum m_l  }\leq\min_{|z|\geq \frac{3}{2}} |P|.        $$
Combining the last two inequalities, we obtain (\ref{mmore}). The
claim is shown. Let's return to (\ref{than}), we get
\begin{eqnarray}
\frac{\min_{|z|\leq \frac{1}{2}} |P|^{-2}} {\max_{|z|\geq
\frac{3}{2}}|P|^{-2}}&\leq& \frac{e^{d_5\lambda}\int_{|z|\leq
\frac{5}{2}}|w|^2 }{C_{11}\int_{|z|\leq \frac{1}{2}}|w|^2 }\nonumber
\\&\leq& e^{d_6\lambda},
\end{eqnarray}
where we applied doubling estimates in the last inequality. Thanks to
(\ref{more}), we obtain
$$  \sum m_l\leq d_7\lambda.     $$
Since $n_l=m_l+1\leq 2m_l$, we complete the lemma for $n_l\geq 2.$

If the vanishing order for the singular point is two, i.e. $n_l=1$.
We consider $Q(z)=\prod(z-z_l)^{\frac{n_l}{2}}$ instead of $P(z)$.
In this case, $Q(z)$ may not be defined as a single valued
holmorphic function on $\mathcal {C}$. We pass to a finite branched
cover of the disk $\mathcal {D}$ punctured at $z_l$. The Carleman
estimates in previous sections still work. The same conclusion will
follow.

\end{proof}

Based on the vanishing order estimate in proposition
\ref{pro2}, we are able to count the number of singular points.

\begin{proof}[Proof of Theorem \ref{th2}]
We cover the double manifold $\overline{\mathcal{M}}$ by geodesic
balls with radius $C\lambda^{-1/2}$. Since $\overline{\mathcal{M}}$
is compact, the order of those balls is $C\lambda$. From proposition
\ref{pro2}, the conclusion in theorem \ref{th2} is arrived.
\end{proof}
\begin{remark}
Thanks to proposition \ref{pro2}, we can actually show a stronger result. Let
$z_l\in \mathcal{M}$ be singular point with vanishing order $n_l+1$,
Then  $\sum_l n_l\leq C\lambda^2$.
\end{remark}

\section{Growth of eigenfunctions}

In this section, we will show that the eigenfunctions do not grow
rapidly on too many small balls. We still restrict $v$ into the
small geodesic ball $\mathbb B(p, c\lambda^{-\frac{1}{2}})$ in the
conformal chart. Let $w(z)= v(c {\lambda^{-\frac{1}{2}}}z)$. Then
$w$ satisfies the elliptic equation (\ref{starr}) with assumptions
(\ref{core2}) in a Euclidean
ball of radius four centered at origin. If we suppose that $w$ grow rapidly, that is,
\begin{equation}
C_{1}\int_{(1-3a)\delta\leq|z-z_l|\leq (1-\frac{3a}{4})\delta}w^2
\leq \int_{(1-\frac{3a}{2})\delta\leq |z-z_l|\leq (1-a)\delta} w^2
\label{rap}
\end{equation}
for all $l$ and some large $C_1$, then the following proposition is valid.
\begin{proposition}
Suppose $D_l$ are disks contained in a Euclidean ball of radius
$\frac{1}{30}$ centered at origin. Furthermore, assume that ($R_1$):
$ \delta< d_1\lambda^{-1}$ and ($R_2$):
$|z_l-z_k|>d_2\lambda^{\frac{1}{2}}\delta$, when $l\not=k$. If
(\ref{rap}) holds for all $l$, the number of disks $D_l$ is less
than $d_3\lambda$. \label{proh}
\end{proposition}
\begin{proof}
  We will use the stronger Carleman estimates in
(\ref{keys1}) in proposition \ref{pro21}. We prove it by contradiction. Suppose that
the collection $D_l=\{z| \,|z-z_l|\leq \delta\}$ are disjoint disks
satisfying the hypotheses $(R_1)$-$(R_3)$ in section 3. Without loss of
generality, we require all the $D_l$ are in a ball centered at
origin with radius $\frac{1}{30}$. As before, $D_l(a)=\{z|\,
|z-z_l|\leq (1-2a)\delta\},$ where $a$ is a suitably small positive
constant. Let $\mathcal {D}$  be a ball centered at origin with
radius 2. We choose a cut-off function $\theta\in
C^{\infty}_{0}(\mathcal{D}\backslash
\cup_l D_l ) $ and assume $\theta(z)$ satisfies the following properties:\\
(i) $\theta(z)=1$, \ if $|z|<1$ and $|z-z_l|>(1-\frac{3}{2}a)\delta$ for all $l$. \medskip\\
(ii) $|\nabla \theta|+|\triangle\theta|<C_2$ if
$|z|>1$. \medskip \\
(iii) $|\nabla \theta|<C_3 \delta^{-1}$,
$|\triangle\theta|<C_4\delta^{-2}$ if
$|z-z_l|<(1-\frac{3}{2}a)\delta$.\\
 Substituting $f=\theta w$ into
(\ref{keys1}) gives that
\begin{eqnarray}
\int_{\mathcal{D}}|\triangle (\theta w)+\hat{b}\cdot\nabla (\theta
w)|^2 |P|^{-2} e^{d_4\lambda |z|^2}& \geq & C_5 \lambda^2
\int_{\mathcal{D}}|
\theta w|^2|P|^{-2} e^{d_4\lambda|z|^2} \nonumber \\
&+&C_{6}\delta^{-4}\int_{A}| \theta w|^2|P|^{-2}e^{d_4\lambda|z|^2}.
\label{keyss}
\end{eqnarray}
We also assume $\tau=d_4\lambda$.
Recall that $A=\cup_l A_l$ and $A_l=\{z| (1-2a)\delta\leq
|z-z_l|\leq (1-a)\delta\}.$ We first consider the integral in the
left hand side of the last inequality. Again, by
(\ref{starr}),
$$\triangle (\theta w) +\hat{b}\cdot\nabla (\theta
w)=-\hat{q}\theta w+\triangle\theta w+2\nabla \theta\cdot\nabla
w+\hat{b}\cdot\nabla \theta w.
$$
Thus,$$|\triangle (\theta w) +\hat{b}\cdot\nabla (\theta w)|^2 \leq
C(c\lambda^2\theta^2 w^2+|\triangle\theta|^2 w^2+|\nabla \theta|^2|\nabla w|^2 +c\lambda|\nabla
\theta|^2 w^2),$$ where $c$ is
sufficiently small. We will absorb the term involving $\theta^2 w^2$
into the right hand side of (\ref{keyss}). Since $c$ is small
enough, we get
\begin{eqnarray}
\int_{\mathcal{D}}(|\triangle \theta|^2 w^2+ c|\nabla \theta|^2
w^2+|\nabla \theta|^2|\nabla w|^2 ) |P|^{-2} e^{d_4\lambda |z|^2}&
\geq & C_7 \lambda^2 \int_{\mathcal{D}}|
\theta w|^2|P|^{-2} e^{d_4\lambda|z|^2} \nonumber \\
&+&C_{8}\delta^{-4}\int_{A}| \theta w|^2|P|^{-2}e^{d_4\lambda|z|^2}.
\end{eqnarray}
Using the properties of $\theta(z)$ and taking into account that all
$D_l$ lies in the ball centered at origin with radius
$\frac{1}{30}$, we obtain
\begin{eqnarray}
\int_{\mathcal{D}}(|\triangle \theta|^2 w^2 + c|\nabla \theta|^2 w^2
&+&|\nabla \theta|^2|\nabla w|^2 ) |P|^{-2} e^{d_4\lambda |z|^2}
\geq C_7 \lambda^2 \int_{\frac{1}{4}\leq |z|\leq\frac{1}{2} }|
w|^2|P|^{-2} e^{d_4\lambda|z|^2}  \nonumber \\
&+&C_{9}\delta^{-4}\sum_{l} \int_{(1-\frac{3a}{2})\delta\leq
|z-z_l|\leq (1-a)\delta}| w|^2|P|^{-2}e^{d_4\lambda|z|^2}.
\label{cont}
\end{eqnarray}
Next we want to control the left hand side of last inequality. Write
\begin{eqnarray}
\int_{\mathcal{D}}|\triangle \theta|^2 w^2 + c|\nabla \theta|^2 w^2
+|\nabla \theta|^2|\nabla w|^2 ) |P|^{-2} e^{d_4\lambda |z|^2}
=I+\sum_l I_l,
\end{eqnarray}
where
$$I=\int_{1\leq|z|\leq 2}|\triangle \theta|^2 w^2 + c|\nabla \theta|^2 w^2
+|\nabla \theta|^2|\nabla w|^2 ) |P|^{-2} e^{d_4\lambda |z|^2},
$$
$$I_l=\int_{(1-2a)\delta\leq|z-z_l|\leq (1-\frac{3a}{2})\delta}|\triangle \theta|^2 w^2 + c|\nabla \theta|^2 w^2
+|\nabla \theta|^2|\nabla w|^2 ) |P|^{-2} e^{d_4\lambda |z|^2}.
$$
By standard elliptic estimates,
\begin{equation}
I\leq e^{d_5\lambda}\max_{|z|\geq 1}|P|^{-2}\int_{\frac{3}{4}\leq
|z|\leq \frac{5}{2}}w^2.
\end{equation}
Similarly, via elliptic estimates,
\begin{equation}
I_l\leq
C_{10}\delta^{-4}(\max_{A_l}|P|^{-2}e^{d_4\lambda|z|})\int_{(1-3a)\delta\leq|z-z_l|\leq
(1-\frac{3a}{4})\delta}w^2.
\end{equation}
Thanks to lemma \ref{lemc},
\begin{equation}
I_l\leq
C_{11}\delta^{-4}(\min_{A_l}|P|^{-2}e^{d_4\lambda|z|})\int_{(1-3a)\delta\leq|z-z_l|\leq
(1-\frac{3a}{4})\delta}w^2.
\end{equation}
Combining those inequalities together in (\ref{cont}) leads to
\begin{eqnarray}
e^{d_5\lambda}\max_{|z|\geq 1}|P|^{-2}\int_{\frac{3}{4}\leq |z|\leq
\frac{5}{2}}w^2&+&C_{11}\delta^{-4}\sum_l(\min_{A_l}|P|^{-2}e^{d_4\lambda|z|})\int_{(1-3a)\delta\leq|z-z_l|\leq
(1-\frac{3a}{4})\delta}w^2 \nonumber \\
&\geq& C_{12} \min_{|z|\leq \frac{1}{2}}|P|^{-2}
\int_{\frac{1}{4}\leq |z|\leq\frac{1}{2} }|
w|^2   \nonumber \\
&+&C_{13}\delta^{-4}\sum_{l} \min_{A_l}(|P|^{-2}e^{d_4\lambda|z|^2})
\int_{(1-\frac{3a}{2})\delta\leq |z-z_l|\leq (1-a)\delta} w^2.
\label{end}
\end{eqnarray}
Performing the similar arguments as (\ref{more}) shows that
$$\min_{|z|\leq \frac{1}{2}}|P|^{-2}>\max_{|z|\geq 1}|P|^{-2} e^{d_5\sum_l m_l}.             $$
If the number of $D_l$ is $d_3\lambda$, then
\begin{equation}
\min_{|z|\leq \frac{1}{2}}|P|^{-2}>\max_{|z|\geq 1}|P|^{-2}
e^{d_6\lambda}. \label{ppp}
\end{equation}
We claim that
\begin{equation}
e^{C_{14}\lambda}\int_{\frac{1}{4}\leq |z|\leq \frac{1}{2}}w^2\geq
\int_{\frac{3}{4}\leq |z|\leq \frac{5}{2}}w^2. \label{cla}
\end{equation}
We prove the claim by doubling estimates shown in proposition
\ref{proo}. We choose  a ball  $\mathbb B(x_0, \frac{1}{8})\subset
\{z|\frac{1}{4}\leq |z|\leq \frac{1}{2}\}$. It is clear that
$$ \int_{\frac{1}{4}\leq |z|\leq \frac{1}{2}}w^2\geq \int_{\mathbb B(x_0, \frac{1}{8})}w^2.  $$
Using doubling estimates, we have
$$ e^{C_{15}\lambda} \int_{\mathbb B(x_0, \frac{1}{8})}w^2\geq \int_{\mathbb B(x_0, \frac{2}{8})}w^2. $$
By finite iterations, we can find a large ball $\mathbb B(x_0, 3)$ that
contains $\{z|\frac{3}{4}\leq |z|\leq \frac{5}{2}\}$. It yields that
$$\int_{ \mathbb B(x_0, 3)}w^2\geq \int_{\frac{3}{4}\leq |z|\leq \frac{5}{2}}w^2.  $$
Then the combination of those inequalities verify the claim.

If we choose $d_3$ is suitably large, since the number disk $D_l$
is $d_3\lambda$, then $d_6$ is suitably large. From the
inequalities (\ref{ppp}) and (\ref{cla}), it follows that
\begin{equation}
e^{d_5\lambda}\max_{|z|\geq 1}|P|^{-2}\int_{\frac{3}{4}\leq |z|\leq
\frac{5}{2}}w^2< C_{12}
\min_{|z|\leq\frac{1}{2}}|P|^{-2}\int_{\frac{1}{4}\leq |z|\leq
\frac{1}{2}}w^2 . \label{get}
\end{equation}
It contradicts the estimates (\ref{rap}) and
(\ref{end}). The proposition is arrived.
\end{proof}

\section{Growth Estimates and Nodal Length}

This section is to find the connection between growth of
eigenfunctions and nodal length. A suitable small growth in $L^2$
norm implies an upper bound of nodal length. We consider the
second order elliptic equations
\begin{equation}
\triangle  \bar w+{b}^\ast \cdot \nabla  \bar w +{q}^\ast  \bar w=0
\quad \mbox{in} \ \mathbb B(0,\, 4). \label{lll}
\end{equation}
Assume that there exist a positive constant $C$ such that
$\|{b}^\ast\|_{W^{1,\infty}}\leq C $ and
$\|{q}^\ast\|_{W^{1,\infty}}\leq C$. The following lemma relies on
the Carleman estimates in lemma \ref{carl}. Suppose $\eps_1$ is a
sufficiently small positive constant.

\begin{lemma}
Suppose that $w$ satisfies the growth estimate
\begin{equation}
\int_{(1-\frac{3a}{2})\eps_0<r<(1-{a})\eps_0}\bar w^2\leq C_3
\int_{(1-{3a})\eps_0<r<(1-\frac{4a}{3})\eps_0} \bar w^2, \label{hyp}
\end{equation}
where $a$ and $\eps_0$ are fixed small constants. Then for
$0<\eps_1<\frac{\eps_0}{100}$, we have \begin{equation}\max_{r\leq
\eps_1}|\bar w|\geq C_4(\frac{\eps_1}{\eps_0})^{C_5}
\big(\dashint_{\mathbb B(0, (1-\frac{4}{3}a)\eps_0)}\bar w^2\big)^{1/2},
\label{god}\end{equation} where $\dashint$ denotes the average of
the integration. \label{prof}
\end{lemma}
\begin{proof}
We select a radial cut-off function $\theta\in
C^\infty_0(\frac{\eps_1}{2}<r<(1-\frac{11a}{10})\eps_0)$ satisfies the properties:\\
(i) $\theta(r)=1$ for
$\frac{3\eps_1}{4}<r<(1-\frac{10a}{9})\eps_0$. \medskip\\
(ii) $|\nabla\theta|+|\triangle \theta|\leq C_6 $ for
$r>(1-\frac{10a}{9})\eps_0$. \medskip\\
(iii) $|\nabla \theta|\leq C_7\eps_1^{-1}$, $|\triangle
\theta|<C_8\eps_1^{-2}$ for $r\leq \frac{3\eps_1}{4}$. \\
 From the
equation (\ref{lll}), we get
$$\triangle (\theta \bar w)+{b}^\ast\cdot \nabla (\theta \bar w)+{q}^\ast\theta \bar w=\triangle \theta \bar w+2\nabla \theta \cdot\nabla \bar  w+
{b}^\ast\cdot\nabla \theta \bar w.$$ Assume that $\tau>C$ is large
enough. Substituting $u=\theta \bar w$ in lemma \ref{carl} yields
that
\begin{equation}
C_2 \tau^3\int r^{\eps} e^{2\tau\phi(r)} \theta^2 \bar w^2  \,drd\omega\leq I,
\label{nea}
\end{equation}
where
$$I=\int r^4  e^{2\tau\phi(r)}|\triangle \theta \bar w+2\nabla \theta \cdot\nabla \bar w+
{b}^\ast\cdot\nabla \theta \bar w|^2\,drd\omega.$$ Note that
$\phi(r)$ is a decreasing function.  Furthermore, by the assumptions
of $\theta(z)$, we obtain
\begin{eqnarray}
I&\leq&
e^{2\tau\phi(\frac{\eps_1}{2})}\int_{\frac{\eps_1}{2}<r<\frac{3\eps_1}{4}}|\triangle
\theta \bar w+2\nabla \theta \cdot\nabla \bar w+ {b}^\ast\cdot\nabla
\theta \bar
w|^2r\,drd\omega  \nonumber \\
&+&
e^{2\tau\phi((1-\frac{10a}{9})\eps_0)}\int_{(1-\frac{10a}{9})\eps_0<r<(1-\frac{11a}{10})\eps_0}|\triangle
\theta \bar w+2\nabla \theta \cdot\nabla \bar w+ {b}^\ast\cdot\nabla
\theta \bar w|^2r\,drd\omega. \nonumber
\end{eqnarray}
By standard elliptic estimates, we derive that
\begin{equation}
I\leq
C_9e^{2\tau\phi(\frac{\eps_1}{2})}\int_{\frac{\eps_1}{4}<r<{\eps_1}}\bar
w^2 r\,drd\omega+
C_{10}e^{2\tau\phi((1-\frac{10a}{9})\eps_0)}\int_{(1-\frac{3a}{2})\eps_0<r<(1-a)\eps_0}
\bar w^2 r\,drd\omega.
\end{equation}
Taking the inequality (\ref{nea}) and assumptions of $\theta$ into
account, we have
\begin{eqnarray}
&C_{10}&e^{2\tau\phi((1-\frac{10a}{9})\eps_0)}\int_{(1-\frac{3a}{2})\eps_0<r<(1-a)\eps_0}\bar
w^2
r\,drd\omega+C_9e^{2\tau\phi(\frac{\eps_1}{2})}\int_{\frac{\eps_1}{4}<r<{\eps_1}}\bar
w^2 r\,drd\omega \nonumber \\ &\geq& C_2
\tau^3\int_{\frac{3\eps_1}{4}<r<(1-\frac{10a}{9})\eps_0} r^{\eps}
e^{2\tau\phi(r)} \bar w^2\,drd\omega \nonumber \\
 &\geq& C_2
\tau^3\big((1-\frac{10a}{9})\eps_0\big)^{\eps-1}
\int_{\frac{3\eps_1}{4}<r<(1-\frac{10a}{9})\eps_0}e^{2\tau\phi(r)}\bar
w^2r\,drd\omega. \label{nnn}
\end{eqnarray}
Since $\eps, \eps_0$ are fixed positive constants,
Taking $\tau$ large enough, we obtain
$$ \frac{1}{2}C_2
\tau^3\big((1-\frac{10a}{9})\eps_0\big)^{\eps-1}>C_{10}.    $$
Taking the hypothesis (\ref{hyp}) into consideration, we can
incorporate the first term in the left hand side of inequality (\ref{nnn})
into the right hand side. It follows that
\begin{equation}
C_9e^{2\tau\phi(\frac{\eps_1}{2})}\int_{\frac{\eps_1}{4}<r<{\eps_1}}\bar
w^2 r\,drd\omega\geq C_{10}
e^{2\tau\phi((1-\frac{10a}{9})\eps_0)}\int_{\frac{3\eps_1}{4}<r<(1-\frac{10a}{9})\eps_0}\bar
w^2r\,drd\omega.
\end{equation}
Fix such $\tau$, adding the following term
$$
e^{2\tau\phi((1-\frac{10a}{9})\eps_0)}\int_{r<\frac{3\eps_1}{4}}\bar
w^2r\,drd\omega
$$
to both sides of last inequality yields that
\begin{equation}
e^{2\tau\phi(\frac{\eps_1}{2})}\int_{r<{\eps_1}}\bar w^2
r\,drd\omega\geq C_{11}
e^{2\tau\phi((1-\frac{10a}{9})\eps_0)}\int_{r<(1-\frac{4a}{3})\eps_0}\bar
w^2r\,drd\omega,
\end{equation}
where we have used the fact that $\phi$ is decreasing.
Straightforward calculations show that
$$ e^{2\tau\big(\phi((1-\frac{10a}{9})\eps_0)-\phi(\frac{\eps_{1}}{2})\big)}  \geq C_{13} (\frac{\eps_1}{\eps_0})^{C_{12}}.  $$
Thus,
\begin{equation}
\int_{r<{\eps_1}}\bar w^2 r\,drd\omega\geq
C_{13}(\frac{\eps_1}{\eps_0})^{C_{12}}\int_{r<(1-\frac{4a}{3})\eps_0}\bar
w^2r\,drd\omega.
\end{equation}
This completes the lemma.
\end{proof}

Our next goal is to find the relation of lemma \ref{prof} and nodal length.
We assume that  the estimate (\ref{hyp}) exists. Then the conclusion
(\ref{god}) in lemma \ref{prof} holds. For $\eps_1\leq
\frac{\eps}{100}$, if $|z|<\eps_1$, using Taylor's expansion gives
that
$$ |\bar w(z)-\sum_{|\alpha|\leq C_{5}}\frac{1}{\alpha !}\frac{\partial^{\alpha} \bar w}{\partial z^{\alpha}}(0)z^\alpha
|\leq \sup_{|z|\leq
\eps_1}\sup_{|\alpha|=C_{5}+1}d_1|\frac{\partial^{\alpha} \bar
w}{\partial z^{\alpha}}(z)|{\eps_1}^{C_{5}+1}, $$ where
$\alpha=(\alpha_1, \alpha_2)$ and $\frac{\partial}{\partial
z^\alpha}=\frac{\partial}{\partial
{z_1}^{\alpha_1}}\frac{\partial}{\partial {z_2}^{\alpha_2}}$. To
control the right hand side of the last inequality, by elliptic
estimates and rescaling argument, we have
$$  |\bar w(z)-\sum_{|\alpha|\leq C_{5}}\frac{1}{\alpha !}\frac{\partial^{\alpha} \bar w}{\partial z^{\alpha}}(0)z^\alpha
|\leq d_2(\dashint_{\mathbb B(0, (1-\frac{4}{3}a)\eps_0)}\bar
w^2)^{1/2}(\frac{\eps_1}{\eps_0})^{C_{5}+1}.
$$
Using the estimate (\ref{god}) in  lemma \ref{prof}, we get
$$ |\bar w(z)-\sum_{|\alpha|\leq C_{5}}\frac{1}{\alpha !}\frac{\partial^{\alpha} \bar w}{\partial z^{\alpha}}(0)z^\alpha
|\leq d_3(\frac{\eps_1}{\eps_0})  \max_{|z|\leq \eps_1}|\bar w|.
$$
Choosing ${\eps_1}/{\eps_0}$ sufficiently small, by the triangle
inequality, we obtain
$$\sup_{|\alpha|\leq C_{5}}\mid\frac{\partial^{\alpha} \bar w}{\partial
z^{\alpha}}(0)\mid{\eps_1}^{|\alpha|}\geq d_4 \max_{|z|\leq
\eps_1}|\bar w|.$$ Applying again the estimate (\ref{god}) to the
right hand side of the last inequality yields that
\begin{equation}
\sup_{|\alpha|\leq C_{5}} \mid \frac{\partial^{\alpha} \bar w}{\partial
z^{\alpha}}(0)\mid{\eps_0}^{|\alpha|}\geq d_5 \big(\dashint_{\mathbb B(0,
(1-\frac{4}{3}a)\eps_0)}\bar w^2\big)^{1/2}. \label{stan}
\end{equation}
By standard elliptic estimates, we also have \begin{equation}
\sup_{|z|\leq \frac{\eps_0}{2}}\sup_{|\alpha|\leq C_{5}+1}
\mid \frac{\partial^{\alpha} \bar w}{\partial
z^{\alpha}}(z)\mid {\eps_0}^{|\alpha|}\leq d_6 \big(\dashint_{\mathbb B(0,
(1-\frac{4}{3}a)\eps_0)}\bar w^2\big)^{1/2}. \label{stan1}
\end{equation}

The basic relationship between derivatives and nodal length in two
dimensions is shown in \cite{DF2}.
\begin{lemma}
Suppose that $\bar w$ satisfies (\ref{stan}) and (\ref{stan1}). Then
$$H^1(z| \, |z|\leq d_7\bar\eps \ \mbox{and} \ \bar w(z)=0)\leq d_8\bar\eps.  $$
\end{lemma}
With aid of the last lemma, we can readily obtain an upper nodal
length estimate.

\begin{proposition}
Let $\bar w$ be the solution of (\ref{lll}). Suppose that $\bar
\eps\leq \eps_0$ and $w$ satisfies the growth condition
\begin{equation}
\int_{(1-\frac{3a}{2})\bar \eps<r<(1-{a})\bar \eps}\bar w^2\leq C_3
\int_{(1-{3a})\bar \eps<r<(1-\frac{4a}{3})\bar \eps} \bar w^2.
\label{aid}
\end{equation}
Then
$$H^1(z|\, |z|\leq d_9\bar\eps \ \mbox{and} \ \bar w(z)=0)\leq d_{10}\bar\eps.  $$
\label{leng}
\end{proposition}
\begin{proof}
Since the inequalities (\ref{stan}) and (\ref{stan1}) can be derived from (\ref{aid}) by lemma \ref{prof}, the proposition follows
from last lemma.

\end{proof}

\section{Total Nodal Length}
As the proposition \ref{proh} indicates that the eigenfunctions can
not grow rapidly on too many small balls. If it grows slowly, we
have an upper bound on the local length of nodal sets by proposition \ref{leng}. In this
section, we will link these two arguments together. To achieve it,
we will employ a process of repeated subdivision and selection
squares. The idea is inspired by \cite{DF2}.

Assume that $\mathbb B(p, c\lambda^{-\frac{1}{2}})$ is a geodesic
ball of double manifold $\overline{\mathcal{M}}$. Choosing $c$ to be
small, then it is contained in a conformal chart. Let $w(z)= v(c
{\lambda^{-\frac{1}{2}}}z)$ with $c$ sufficiently small. We know
that $w$ satisfies
\begin{equation}
\triangle w+\hat{{b}}(x)\cdot\nabla w+\hat{{q}}(x)w=0 \quad
\mbox{in}\ \mathbb B( 0, 4).
\end{equation}

 We consider the square $P=\{ (x,
y)| \max(|x|, |y|)\leq \frac{1}{60}\}$ in $\mathbb B (0, 4)$ and
divide it into a grid of closed square $P_l$ with side $ \delta\leq
a_1\lambda^{-1}$. If (\ref{rap}) holds for some point $z_l\in P_l$
and for some sufficiently large $C_{1}$. We call $P_l$ as a square
of rapid growth. With aid of proposition \ref{proh}, we are able to obtain the following result.

\begin{lemma}
There are at most $C\lambda^2$ squares with side $ \delta$ where
$w$ is of rapid growth. \label{fina}
\end{lemma}
\begin{proof}
Let $I_1$ be the collection of those indices $l$ for which $P_l$ is
a square of rapid growth. For each $l\in I_1$, there exists some
point $z_l\in P_l$ such that (\ref{rap}) holds.  Let $|I_1|$ denote
the cardinality of $I_1$. Define
$$ P_l^{\ast}=\{z|\, |z-z_l|<d_1\delta  \lambda^{\frac{1}{2}}\}.  $$
The collection of disks $P_l^{\ast}$ covers the collection of square
$P_l$ for $l\in I_1$. We choose a maximal collection of disjoint
disks of $P_l^{\ast}$ and denote it as $I_2$. If $l\in I_2$, we
define
$$ P_l^{\ast\ast}=\{z| \,|z-z_l|<4d_1\delta  \lambda^{\frac{1}{2}}\}.  $$
Since the collection of disks in $I_2$ are disjoint and maximal, we
obtain that
$$ \bigcup_{l\in I_2}P_l^{\ast\ast}\supseteq \bigcup_{l\in I_1}P_l^{\ast} \supseteq \bigcup_{l\in I_1}P_l.      $$
Thus,
$$|I_2|\times 16d_1^2\delta^2\lambda\geq |I_1|\delta^2,     $$
which implies
$$|I_2|\lambda\geq d_2 |I_1|. $$
Recall from proposition \ref{proh} that $|I_2|\leq d_3\lambda$.
Therefore, we obtain the desirable estimate $|I_1|\leq
d_4\lambda^2$.

\end{proof}

Now we introduce a iterative process of bisecting squares. We begin
by dividing the square into a grid of square $P_l(1)$ with side
$\delta(1)=a_1\lambda^{-1}$, then separate them into two categories
$R_l(1)$ and $S_l(1)$. $R_l(1)$ are those where $w$ is of rapid
growth and $S_l(1)$ are those where (\ref{rap}) fails for $w$. We
continue to bisect each square $R_l(1)$ to obtain square $P_{l}(2)$
with side $\delta(2)=\frac{\delta(1)}{2}$. Again, we split $P_i(2)$
into the subcollection $R_l(2)$ with rapid growth and $S_l(2)$ with
slow growth. We repeat the process at each step $k$. Then there are
squares $R_l(k)$ and $S_l(k)$ with
$\delta(k)=\frac{\delta(1)}{2^k}$. We count the number of $R_l(k)$
and $S_l(k)$ at step $k$.
\begin{lemma}
(i) The number of squares $R_l(k)$ is at most $C_2\lambda^2$.\\
(ii) The number of squares $S_l(k)$ is at most $C_3\lambda^2$.
\label{mei}
\end{lemma}
\begin{proof}
The conclusion (i) follows directly from the lemma \ref{fina}. We
only need to show (ii). If $k=1$, the conclusion (ii) follows
because the total number of squares is at most the order of
$\lambda^2$. If $k\geq 2$, by construction of those squares,
$$ |S_l(k)|\leq 4|R_l(k-1)|\leq C_4\lambda^2,        $$
where we have used (i) in the last inequality. The lemma is done.
\end{proof}
Next lemma tells that almost every point lies in some $R_l(k)$ with
slow growth. It is the lemma 6.3 in \cite{DF2}.
\begin{lemma}
 $\bigcup_{k,l} S_l(k)$ covers the square $P$ except for singular points $\mathcal {S}=\{z\in P| w(x)=0, \nabla w=0\}$.
 \label{singular}
\end{lemma}

We are ready to give the proof of Theorem \ref{th1}.
\begin{proof}[Proof of Theorem \ref{th1}]
Consider $\bar w(z)= w(z_l+\eps_0^{-1}\delta(k)z)$. Then $\bar w(z)$
satisfies the equation (\ref{lll}). Choosing a finite collection of
$z_l\in S_l(k) $ and applying proposition \ref{leng}, we have
\begin{equation}
H^1(z| \, w(z)=0 \ \mbox{and} \ z\in S_l(k) )\leq C_5 2^{-k}
\lambda^{-1}. \label{final}
\end{equation}
Furthermore, thanks to lemma \ref{singular},
\begin{eqnarray}
H^1(z| \, w(z)=0 \ \mbox{and} \ \max(|x|, |y|)\leq \frac{1}{60}
)&\leq & \sum_{l, k} H^1(z|
\, w(z)=0 \ \mbox{and} \ z\in S_l(k) ) \nonumber \\
&\leq & \lambda^2 \sum_k C_5 2^{-k} \lambda^{-1} \nonumber\\
&\leq & C_6\lambda,
\end{eqnarray}
where we have used (ii) in lemma \ref{mei} and (\ref{final}). Since
 $w(z)=v(c\lambda^{-\frac{1}{2}}z)$, by the rescaling argument, we
 obtain
 $$ H^1( \{v(z)=0\}\cap \mathbb B(p, c\lambda^{-\frac{1}{2}}))\leq C_6\lambda^{\frac{1}{2}}.   $$
 Finally, covering $\overline{\mathcal{M}}$ with order $\lambda$ of geodesic
 balls with radius $c\lambda^{-\frac{1}{2}}$, we readily deduce
 that
$$ H^1(z\in  \overline{\mathcal{M}} |\, v(z)=0) \leq C_7\lambda^{\frac{3}{2}}. $$
So is $H^1(\mathcal{N}_\lambda)$.
\end{proof}

\end{document}